\documentclass{article}

\usepackage{amsmath,amsfonts,latexsym,amssymb,amsthm,url}

\newtheorem{theorem}{Theorem}[section]
\newtheorem{proposition}[theorem]{Proposition}
\newtheorem{corollary}[theorem]{Corollary}
\newtheorem{lemma}[theorem]{Lemma}
\newtheorem{remark}[theorem]{Remark}
\newtheorem{problem}[theorem]{Problem}

\DeclareMathOperator{\height}{h}
\DeclareMathOperator{\im}{Im}
\DeclareMathOperator{\re}{Re}

\newcommand{\C}{\mathbb{C}}
\newcommand{\Q}{\mathbb{Q}}
\newcommand{\R}{\mathbb{R}}
\newcommand{\Z}{\mathbb{Z}}

\newcommand{\CC}{\mathcal{C}}
\newcommand{\OO}{\mathcal{O}}

\newcommand{\gerp}{\mathfrak{p}}

\newcommand{\barv}{{\bar{v}}}

\title{Skolem Problem for Linear Recurrence Sequences with~$4$ Dominant Roots\\[0.5cm]
\large after Mignotte, Shorey, Tijdeman, Vereshchagin and Bacik}

\author{Yuri BILU\footnote{Supported by the ANR project JINVARIANT}}

\date{\today}

\numberwithin{equation}{section}

\begin{document}
\hfuzz=4pt

\maketitle

\begin{abstract}
It is an exposition of the work of  Mignotte, Shorey, Tijdeman, Ve\-resh\-cha\-gin and Bacik on decidability of the vanishing problem for linear recurrence sequences of order~$4$. 
\end{abstract}

 {\footnotesize
 
\tableofcontents

}

\section{Introduction}

A map ${U:\Z_{\ge 0}\to \bar \Q}$ is called a Linear Recurrence Sequence (LRS in the sequel) or order~$r$ if there exist algebraic number ${a_0, \ldots, a_{r-1}\in \bar \Q}$ 
$$
U(n+r)=a_{r-1}U(n+r-1) +\cdots +a_0U(n) \qquad (n \in \Z_{\ge 0}). 
$$
The order~$r$ is called \textit{minimal} if~$U$ is an LRS of order~$r$, but not of order ${r-1}$. In the sequel, we will always assume that the order is minimal. In particular, this implies that the coefficients ${a_0, \ldots, a_{r-1}}$ are uniquely defined, and that ${a_0\ne 0}$. 

The polynomial ${\chi(T)=T^r-a_{r-1}T^{r-1}-\cdots-a_0\in \bar\Q[T]}$ is called the \textit{characteristic polynomial} of the LRS~$U$. Let ${\lambda_1, \ldots, \lambda_s\in\bar\Q }$ be the distinct roots of~$\chi(T)$. They will also be called the \textit{roots} of~$U$. Since ${a_0\ne 0}$, we have ${\lambda_1, \ldots, \lambda_s \ne 0}$. 

Let
${\chi(T) =(T-\lambda_1)^{r_1}\cdots (T-\lambda_s)^{r_s}}$
be the factorization of the characteristic polynomials. Then $U(n)$ can be expressed as 
\begin{equation}
\label{ebinet}
U(n) =f_1(n) \lambda_1^n + \cdots+ f_s(n) \lambda_s^n, 
\end{equation}
where  ${f_1(T), \ldots, f_s(T) \in \bar\Q[T]}$ are polynomials; if the order is minimal, their degrees  are given by 
${\deg f_j=r_j-1}$. 
Expression~\eqref{ebinet} is known as the \textit{Binet formula}.

The classical Skolem-Mahler Theorem asserts that a non-degenerate LRS may have at most finitely zeros; that is, equation ${U(n)=0}$ has at most finitely many solutions in  ${n\in \Z_{\ge0}}$. However, all known proofs of this theorem are not constructive, and decidability of  the following two famous problems are widely open.

\begin{problem}[Weak Skolem Problem]
Decide whether a given  non-degenerate LRS vanishes; that is, whether its set of zeros is non-empty.    
\end{problem}

\begin{problem}[Strong Skolem Problem]
Find the full set of zeros of a given non-degenerate LRS. 
\end{problem}

Clearly, decidability of the Strong Skolem  Problem 
would imply that of the Weak Skolem Problem. 
It is known~\cite{BLNOPW22} that decidability of the Strong Skolem Problem for simple\footnote{The LRS~$U$ is called \textit{simple} if ${r_1=\cdots=r_s=1}$, that is,  ${\lambda_1, \ldots, \lambda_s \ne 0}$ are simple roots of the characteristic polynomial $\chi(T)$; equivalently, the LRS is simple if ${r=s}$. In this case, polynomials $f_1, \ldots, f_s$ from the Binet formula are constant. }
 LRS  
follows from some widely believed conjectures. 

Most of the unconditional results in favor of decidability of the Strong Skolem Problem 
for some classes of LRS use the technique of dominant roots.  Let~$K$ be a number field containing the roots ${\lambda_1, \ldots, \lambda_s}$, and let ${v\in M_K}$ be an absolute value of~$K$. (See Subsection~\ref{ssav} for our conventions on absolute values.) Let us re-number the roots to satisfy 
$$
|\lambda_1|_v\ge |\lambda_2|_v \ge \cdots \ge |\lambda_s|_v.
$$
We say that the LRS $U$ has a \textit{dominant root} (with respect to~$v$) if 
$$
|\lambda_1|_v> |\lambda_2|_v \ge \cdots \ge |\lambda_s|_v.
$$
In  this case it is obvious that
$$
|f_1(n) \lambda_1^n|_v  > |f_2(n) \lambda_2^n + \cdots+ f_s(n) \lambda_s^n|_v
$$
for ${n\ge n_0}$, where ${n_0>0}$ can be expressed explicitly in terms of~$U$. Hence ${U(n)\ne 0}$ for ${n\ge n_0}$. We have proved the following.

\begin{proposition}
\label{prtriv}
The  Strong Skolem Problem is decidable for a non-degenerate LRS having a dominant root with respect to some ${v\in M_K}$. 
\end{proposition}

Proposition~\ref{prtriv} is trivial, but it admits a quite non-trivial generalization. 
Let ${k\in \{1, \ldots, s\}}$. We say that the LRS $U$ has \textit{exactly~$k$ dominant roots} (with respect to~$v$) if 
$$
|\lambda_1|_v=\cdots= |\lambda_k|_v >|\lambda_{k+1}|_v\ge \cdots \ge |\lambda_s|_v. 
$$
We say that it has \textit{at most~$k$ dominant roots} if it has exactly~$\ell$ dominant roots for some ${\ell\le k}$.  Equivalently,~$U$has at most~$k$ dominant roots if ${|\lambda_1|_v>|\lambda_{k+1}|_v}$.

The following theorem was proved independently by  Mignotte, Shorey and Tijdeman~ \cite{MST84} and by Vereshchagin~ \cite{Ve85}. 

\begin{theorem}
\label{thmstv}
Let~$U$ be a non-degenerate LRS and~$K$ a number field containing its roots. Assume that one of the following is true:
\begin{itemize}
\item
the LRS~$U$ has most~$2$ dominant roots with respect to  some ${v \in M_K}$, 
\item
or~$U$ has exactly~$3$ dominant roots with respect to some \textbf{infinite} ${v\in M_K}$. 
\end{itemize}
Then the Strong Skolem Problem is decidable for $U$. 
\end{theorem}

Here is an immediate consequence. 

\begin{corollary}
\label{cmstv}
The Strong Skolem Problem is decidable for  a non-degenerate LRS with at most~$3$ distinct roots (that is, with ${s\le 3}$). In particular, it is decidable for LRS of order at most~$3$. 
\end{corollary}

Corollary~\ref{cmstv} was recently refined by Bacik~\cite{Ba24}.

\begin{theorem}
\label{thba}
The Strong Skolem Problem is decidable for  a non-degenerate LRS with at most~$4$ distinct roots; in particular, it is decidable for LRS of order at most~$4$. 
\end{theorem}

In \cite{MST84,Ve85} Theorem~\ref{thba} was proved under the additional hypothesis ${U(n)\in \R}$ for all~$n$.

This note  is an exposition of the proofs of Theorems~\ref{thmstv} and~\ref{thba}. We insist that all the mathematical ideas already appeared elsewhere, notably,  in \cite{Ba24,MST84,Ve85}. Our contribution is purely expository, if any.

\section{Preparatory material}

In Subsections~\ref{ssav} and~\ref{sshe} we give a very brief synopsis of absolute values and heights; more details can be found in standard texts like \cite{BG06,HS00,La83}. In Subsection~\ref{sslog} we introduce our principal tool, Baker's Inequality. Finally, in Subsection~\ref{sstri} we prove an elementary geometric inequality used later in Section~\ref{sthree}. 

\subsection{Absolute values}
\label{ssav}

On the field~$\Q$
of rational numbers there are the ``usual'' absolute value ${|\cdot|}$ and the \textit{$p$-adic} absolute values for each prime number $p$, defined by 
${|x|_p := p^{-\nu_p(x)} }$ for ${x\in \Q^\times}$.
It is common to call the ``usual'' absolute value \textit{infinite absolute value}, and denote it by ${|\cdot|_\infty}$; informally, it is the $p$-adic absolute value for the ``infinite prime'' ${p=\infty}$. 

Denote by~$M_\Q$ the set consisting of all prime numbers and the symbol~$\infty$:
$$
M_\Q:=\{2,3,5,7,11, \dots\}\cup \{\infty\}. 
$$
According to the \textit{Theorem of Ostrowski}, every (non-trivial) absolute value on~$\Q$ is equivalent to a $p$-adic absolute value for some ${p\in M_\Q}$. Note also the \textit{Product Formula}
$$
\prod_{p\in M_\Q}|x|_p=1 \qquad (x\in \Q^\times). 
$$

Now let~$K$ be a number field and $|\cdot|_v$ a (non-trivial) absolute value on~$K$.  By the Theorem of Ostrowski, its restriction to~$\Q$ is equivalent to $|\cdot|_p$ for some ${p\in M_\Q}$; this is denoted as ${v\mid p}$. We normalize $|\cdot|_v$ so that its restriction to~$\Q$ coincides with $|\cdot|_p$; that is, ${|p|_v=p^{-1}}$ if ${v\mid p<\infty}$, and ${|2025|_v=2025}$ if ${v\mid \infty}$.

Denoting by~$M_K$ the set of non-trivial absolute values on~$K$ normalized as indicated, we have the Product Formula 
\begin{equation}
\label{eprofor}
\prod_{v\in M_K}|x|_v^{d_v}=1\qquad (x \in K^\ast),
\end{equation}
where ${d_v:=[K_v:\Q_p]}$ is the local degree of~$v$.

The absolute value $|\cdot|_v$ is called \textit{infinite} if ${v\mid \infty}$, that is, if ${|\cdot|_v}$ extends the infinite absolute value on~$\Q$. The infinite absolute values on~$K$ stay in one-to-one correspondence with the set of real embeddings ${K\stackrel\sigma\hookrightarrow\R}$, and of pairs of complex conjugate embeddings ${K\stackrel{\sigma,\bar\sigma}\hookrightarrow\C}$. If~$v$ corresponds to a real embedding ${K\stackrel\sigma\hookrightarrow\R}$ or, respectively,  to a pair ${K\stackrel{\sigma,\bar\sigma}\hookrightarrow\C}$, then~$v$ is defined by ${|x|_v:=|\sigma(x)|}$ for ${x\in K}$, and the local degree~$d_v$ is~$1$, respectively~$2$. 

If ${v\mid p}$, where~$p$ is a finite prime, then the absolute value $|\cdot|_v$ is called \textit{finite}. For a given rational prime~$p$ the finite absolute values extending ${|\cdot|_p}$ stay in one-to-one correspondence with the primes of~$K$ dividing~$p$. If~$\gerp$ is  such a prime, then the corresponding~${|\cdot|_v}$, normalized as above,  is given by 
\begin{equation}
\label{evp}
|x|_v :=p^{-\nu_\gerp(x)/e}\qquad (x\in K^\times), 
\end{equation}
where ${e=e_\gerp:=\nu_\gerp(p)}$ is the ramification index. If ${f=f_\gerp}$ is the inertia degree, defined by ${\#(\OO_K/\gerp)=p^f}$, then the local degree is given by  ${d_v=ef}$.

Thus, we may write 
$$
M_K= \{\text{primes of $K$}\} \cup \{K\stackrel\sigma\hookrightarrow\R\}\cup\{K\stackrel{\sigma,\bar\sigma}\hookrightarrow\C\}. 
$$
The following \textit{Kronecker Theorem} is fundamental.

\begin{proposition}[Kronecker Theorem]
Let ${x\in K^\times}$ be such that ${|x|_v\le1}$ for all ${v\in M_K}$. Then~$x$ is a root of unity.
\end{proposition}

\subsection{Heights}
\label{sshe}

Let ${x \in \bar\Q}$ be an algebraic number. Fix a number field~$K$ containing~$x$, and let ${d:=[K:\Q]}$ be its degree. We define the \textit{absolute logarithmic height} (the \textit{height} in the sequel) of~$x$  by
$$
\height(x):= d^{-1}\sum_{v\in M_K}d_v \log\max\{|x|_v, 1\}. 
$$
It is known (and not hard to see) that the right-hand side depends only on~$x$, and not on the particular choice of the field~$K$. 

\begin{remark}
There is an equivalent definition of the height, as the logarithmic Mahler measure of the minimal polynomial, divided by the degree. Equivalence of the two definitions is an easy consequence of the Gauss Lemma; see, for instance, ``Remark''  on the bottom of page~53 of~\cite{La83}. We mention this equivalent definition only because it used in some of the articles quoted in this note, like \cite{Vo96,Yu98}. 
\end{remark}

The height has the following properties (the proofs are easy and can be found in~\cite{BG06,HS00,La83}.). 

\begin{proposition}
\label{prheights}
\begin{enumerate}
\item
\label{iconj}
If ${x,y\in \bar\Q}$ are conjugate over~$\Q$ then ${\height(x)=\height(y)}$. 

\item
\label{isumpro}
For ${x_1, \ldots, x_n\in  \bar\Q}$ we have 
\begin{align*}
\height(x_1+\cdots+x_n) &\le \height(x_1)+\cdots+\height(x_n) +\log n, \\
\height(x_1\cdots x_n) &\le \height(x_1)+\cdots+\height(x_n). 
\end{align*}

\item
\label{ipower}
For ${x\in \bar\Q}$ and ${n\in \Z}$ (with ${x\ne 0}$ if ${n<0}$) we have ${\height(x^n) =|n|\height(x)}$.

\item (Kronecker Theorem) 
\label{ikr}
We have ${\height(x)=0}$ if and only if~$x$ is~$0$ or a root of unity. Moreover, for every positive integer~$d$ there exists ${\kappa(d)>0}$  such that for~$x$ of degree~$d$, non-zero and not a root of unity, we have ${\height(x)\ge \kappa(d)}$ (heights are ``separated'' from~$0$). 

\item
(Northcott Theorem)
For every ${C>0}$ there exist at most finitely many algebraic numbers of height and degree bounded by~$C$. 

\end{enumerate}
\end{proposition}

\begin{remark}
\label{reheights}
\begin{enumerate}
\item
An important special case of Item~\ref{ipower} is 
\begin{equation}
\label{exmone}
\height(x^{-1})=\height(x) \qquad (x\in \bar\Q^\times). 
\end{equation}
In fact,~\eqref{exmone} follows immediately from the  Product Formula~\eqref{eprofor}. 

\item
Item~\ref{ipower} can be generalized as follows. Let ${r\in \Q}$ and ${x\in \bar\Q}$ (with ${x\ne 0}$ if ${r<0}$). Then ${\height(x^r)=|r|\height(x)}$, for any definition of the rational power $x^r$. In particular, 
\begin{equation}
\label{esqrt}
\height(\sqrt x)=\frac12\height(x). 
\end{equation}

\item
Let~$z$ be a complex algebraic number. Then Items~\ref{iconj} and~\ref{isumpro}, together with~\eqref{esqrt},  imply that 
$$
\height(\bar z)=\height(z), \qquad \height(|z|) \le \height(z), \qquad \height(\re z), \height(\im z) \le 2\height(z) +2\log2. 
$$

\item
Various explicit expressions for $\kappa(d)$ from Item~\ref{ikr} can be found in literature. For instance, Voutier \cite[Corollary~2]{Vo96}, elaborating on ideas of Dobrovolski~\cite{Do79}, proved that, for ${d\ge 2}$, 
\begin{equation}
\label{evo}
\kappa(d):=\frac{2}{d(\log(3d))^3}
\end{equation}
would do. Also, we have clearly  ${\kappa(1)=\log2}$. 
Slightly adjusting the right-hand side of~\eqref{evo}, we see that
\begin{equation}
\label{evol}
\kappa(d):= \frac{1}{2d(\log(3d))^3}
\end{equation}
would do for all~$d$. 

Conjecturally, $\kappa(d)$ can be taken as $c/d$, with some positive absolute constant~$c$ (this is known as Lehmer's conjecture). 

\end{enumerate}
\end{remark}

Here is an immediate consequence.

\begin{corollary}
\label{contoh}
Let~$x$ be a non-zero algebraic number of degree~$d$, not a root of unity, and let ${n\in \Z}$. Then, denoting ${y:=x^n}$, we have 
$$
|n|=\frac{\height(y)}{\height(x) }\le \frac{\height(y)}{\kappa(d)}\le 2d(\log(3d))^3\height(y). 
$$
\end{corollary}

The following property is usually called   \textit{Liouville's inequality}. 

\begin{proposition}
\label{prliou}
Let~$K$ be a number field of degree~$d$, and ${x\in K^\times}$. Then for every ${v\in M_K}$ we have 
\begin{equation}
\label{eliou}
e^{-d\height(x)} \le |x|_v^{d_v}\le e^{d\height(x)}. 
\end{equation}
\end{proposition}
Here, the inequality on the right follows immediately from the definition of the height, and that on the left follows by applying the inequality on the right to $x^{-1}$ instead of~$x$.

\subsection{Baker's inequality}
\label{sslog}

Let~$K$ be a number field  of degree~$d$ and ${v\in M_K}$ an absolute value on~$K$.  We need a lower bound for the expression of the form ${|\alpha\lambda^n-1|_v}$, where ${\alpha, \lambda \in K^\times}$ and ${n\in \Z_{\ge0}}$, of course, provided that ${\alpha\lambda^n\ne 1}$. Using Proposition~\ref{prheights}, we estimate 
$$
\height(\alpha\lambda^n-1)\le \height(\alpha)+n\height(\lambda)+\log2, 
$$
which, by Proposition~\ref{prliou}, implies that 
\begin{equation}
\label{eliouour}
|\alpha\lambda^n-1|_v\ge e^{-d(\height(\alpha)+n\height(\lambda)+\log2)}.
\end{equation}
This bound is, however, insufficient for our purposes: we need a lower bound of the form ${|\alpha\lambda^n-1|_v\ge e^{-o(n)}}$ as ${n\to \infty}$.
We are going to prove the following.

\begin{proposition}
\label{prbak}
Let~$K$ be a number field  of degree~$d$ and ${\lambda \in K^\times}$. Let  ${v\in M_K}$ an absolute value on~$K$ with underlying prime ${p\in M_\Q}$. Then there exists ${C>0}$, effectively depending on~$d$,~$p$ and~$\lambda$, such that for any ${\alpha \in K^\times}$  and ${n\in \Z}$ with ${n\ge 2}$,  we have either ${\alpha\lambda^n=1}$ or 
\begin{equation}
\label{ebak}
|\alpha\lambda^n-1|_v \ge e^{-C(\height(\alpha)+1)\log n}. 
\end{equation}
\end{proposition}


This proposition is a special case of a deep result known as \textit{Baker's  inequality}, which provides a non-trivial lower bound for expressions of the form ${|\theta_1^{b_1}\cdots \theta_s^{b_s}-1|_v}$, where $\theta_j$ are non-zero algebraic numbers and $b_j$ are rational integers such that ${\theta_1^{b_1}\cdots \theta_s^{b_s}\ne 1}$.  Various forms of this inequality are available in the literature.  We will use the results of Matveev~\cite{Ma00}\footnote{The earlier work of Baker and Wüstholz~\cite{BW93} would suffice as well.} for  infinite~$v$ and of Yu~\cite{Yu98} when~$v$ is finite. 


The following is Corollary~2.3 of Matveev~\cite{Ma00}. 

\begin{theorem}[Matveev]
\label{thmat}
Let ${\theta_1, \ldots, \theta_s}$ be non-zero complex algebraic numbers belonging to a number field of degree~$d$, and  ${\log\theta_1, \ldots , \log\theta_s}$ some determinations of their complex logarithms.   Let ${b_1, \ldots, b_s\in \Z}$ be such that 
$$
\Lambda:=b_1\log\theta_1+\cdots+b_s\log\theta_s\ne 0. 
$$
 Let ${A_1, \ldots, A_s, B}$ be real numbers satisfying 
\begin{align*}
A_k&\ge \max\{d\height(\theta_k),|\log\theta_k|,0.16\}\qquad (k=1,\ldots, s),\\
B&\ge \max\{|b_1|, \ldots, |b_s|\}. 
\end{align*}
Then 
\begin{equation*}
|\Lambda|\ge e^{-2^{6s+20} d^2(1+\log d)A_1\cdots A_s (1+\log B)}. 
\end{equation*}
\end{theorem}

The case of infinite~$v$ of Proposition~\ref{prbak} is an easy consequence of this theorem. In what follows $\arg$ and $\log$ stand for the principal branches of the complex argument and logarithm; that is, for ${z\in \C^\times}$ we have 
\begin{equation}
\label{elogconv}
-\pi <\arg z = \im\log z \le \pi.
\end{equation}
We will use the following simple lemma.

\begin{lemma}
\label{ldweg}
Let ${0<\rho <1}$ and let ${z\in \C}$ satisfy ${|z-1|\le \rho}$. Then 
$$
|\log z|\le -\frac{\log(1-\rho)}\rho|z-1|
$$
In particular, if ${|z-1|\le 1/2}$ then ${|\log z|\le 1.39|z-1|}$. 
\end{lemma}

\begin{proof}
We have 
$$
 \left|\frac{\log z}{z-1}\right|= \left|\sum_{k=0}^\infty (-1)^k\frac{(z-1)^k}{k+1}\right| \le \sum_{k=0}^\infty \frac{\rho^k}{k+1}= -\frac{\log(1-\rho)}\rho, 
$$
as wanted. 
\end{proof}

\begin{proof}[Proof of Proposition~\ref{prbak} for infinite~$v$]

Since the absolute value~$v$ is infinite, there is an embedding ${K\stackrel\sigma\hookrightarrow \C}$  such that ${|x|_v=|\sigma(x)|}$ for ${x\in K}$. We will identify~$K$ with its image $\sigma(K)$; in particular, $\alpha, \lambda$ are now complex algebraic numbers. Assuming that 
${\alpha\lambda^n\ne 1}$, we have to prove the lower bound 
\begin{equation}
\label{eloboc}
|\alpha\lambda^n-1| \ge e^{-C_1(\height(\alpha)+1)\log n},
\end{equation}
where in this proof $C_1, C_2$, etc. are positive quantities depending (effectively) only on~$d$ and~$\lambda$.  

Let~$m$ be the nearest integer to ${(\arg \alpha + n\arg \lambda)/(2\pi)}$, where we take the smaller one if there are two nearest integers. This means that   ${m\in \Z}$ and 
\begin{equation}
\label{ebetween}
-\pi < |\arg \alpha + n\arg \lambda -2\pi m|\le \pi. 
\end{equation}
Note that ${|2m|\le n+2}$. 

Defining 
$$
\Lambda:= \log \alpha +n\log \lambda -2m\pi i, 
$$
inequality~\eqref{ebetween} implies that, with our convention~\eqref{elogconv}, we have
${\log (\alpha\lambda^n)= \Lambda}$. 
In particular, since ${\alpha\lambda^n\ne 1}$, we have ${\Lambda \ne 0}$.  

We may  assume that ${|\alpha\lambda^n- 1|\le 1/2}$, since otherwise ~\eqref{eloboc} holds trivially. Lem\-ma~\ref{ldweg} implies that ${|\alpha\lambda^n-1|\ge 1.39^{-1}|\Lambda|}$, and it suffices to prove that 
\begin{equation}
\label{elobocla}
|\Lambda| \ge e^{-C_2(\height(\alpha)+1)\log n}. 
\end{equation}
We apply Theorem~\ref{thmat} with the following data: ${s:=3}$, 
\begin{align*}
\theta_1&:= \alpha, & \theta_2&:=\lambda, &\theta_3&:=-1, \\
A_1&: =d\height(\alpha)+\pi, & A_2&:=d\height(\lambda)+\pi, & A_3&:=\pi, \\
b_1&:= 1, & b_2&:=n, &b_3& := -2m, &
B&:=n+2.  
\end{align*}
For a complex algebraic number ${z\in K}$  we have 
${|\log z|\le \log|z|+\pi \le d\height(z)+\pi}$
by Proposition~\ref{prliou}, which implies that our  $A_1,A_2,A_3$ satisfy the hypothesis of Theorem~\ref{thmat}. 

We have 
${A_1A_2A_3 \le C_3(\height(\alpha)+1)}$. Hence~\eqref{elobocla} follows from 
Theorem~\ref{thmat}. 
This proves Proposition~\ref{prbak} for infinite~$v$. 
\end{proof}

To prove Proposition~\ref{prbak} for finite~$v$ we use a result of Yu, see the displayed equation on  top of page~30 of~\cite{Yu98}. 

\begin{theorem}[Yu]
\label{thyu}
Let~$K$ be a number field of degree~$d$ and~$\gerp$ be a prime of~$K$ with underlying rational prime~$p$.  
Let ${\theta_1, \ldots, \theta_s\in K^\times}$ satisfy 
$$
\nu_\gerp(\theta_k) =0 \qquad (k=1, \ldots, s), 
$$ 
and let ${b_1, \ldots, b_s\in \Z}$ be such that ${\Xi:=\theta_1^{b_1}\cdots \theta_s^{b_s}-1\ne 0}$. 
Let ${h_1, \ldots, h_s, B}$ be real numbers satisfying 
\begin{align*}
h_k&\ge \max\{\height(\theta_k),\log p\}\qquad (k=1,\ldots, s),\\
B&\ge \max\{|b_1|, \ldots, |b_s|,3\}. 
\end{align*}
Then 
\begin{equation*}
\nu_\gerp(\Xi) \le C_0h_1\cdots h_s \log B, \qquad
C_0:= 12\left( \frac{6(s+1)d}{(\log p)^{1/2}}\right)^{2(s+1)}p^d\bigl(5+\log(sd)\bigr).
\end{equation*}
\end{theorem}

For our purposes the exact formula for~$C_0$ is not relevant; it is only important that~$C_0$ can be expressed explicitly in terms of~$d$,~$p$ and~$s$. 

Yu improved on his result in subsequent articles, but the improvements mainly concern the shape of the expression for~$C_0$. 

\begin{proof}[Proof of Proposition~\ref{prbak} for finite~$v$]
Let~$\gerp$ be the prime of~$K$ such that~\eqref{evp} holds; that is, for ${x\in K^\times}$ we have ${|x|_v =p^{-\nu_\gerp(x)/e}}$, where ${e=e_\gerp}$ is the ramification index.

We may assume that ${|\alpha\lambda^n-1|_v<1}$, since otherwise~\eqref{ebak} holds trivially. In particular, ${|\alpha\lambda^n|_v=1}$. 

If ${|\lambda|_v \ne 1}$ then ${n=-\log|\alpha|_v/\log|\lambda|_v}$.  Proposition~\ref{prliou} implies that 
$$
|\log|\alpha|_v| \le \frac{d}{d_v}\height(\alpha).
$$
Also, since $ {|\lambda|_v \ne 1}$, we have either ${|\lambda|_v \le p^{-1/e}}$ or ${|\lambda|_v \ge p^{1/e}}$; in other words, 
$$
|\log|\lambda|_v|\ge\frac1e \log p. 
$$
Hence
$$
n \le \frac d{\log p}\frac{e}{d_v}{\height(\alpha)}\le \frac{d}{\log2}\height(\alpha), 
$$
and the result follows from~\eqref{eliouour}. 

Now assume that ${|\lambda|_v=1}$. Then ${|\alpha|_v=1}$ as well. In other words, 
$$
\nu_\gerp(\alpha)=\nu_\gerp(\lambda)=0.
$$
Hence Theorem~\ref{thyu} applies and we obtain 
$$
\nu_\gerp(\alpha\lambda^n-1) \ge C_1(\height(\alpha)+1)\log n, 
$$
where~$C_1$ depends effectively on~$d$,~$p$ and~$\lambda$. Hence 
$$
|\alpha\lambda^n-1|_v =p^{-\nu_\gerp(\alpha\lambda^n-1)/e} \ge e^{-C(\height(\alpha)+1)\log n}, 
$$
where ${C=C_1\log p}$. The proof is complete. 

\end{proof}

\subsection{Two circles}
\label{sstri}

In this subsection we prove an elementary geometric estimate, which is  needed to apply ``Beukers' trick'' in the proof of Proposition~\ref{prbeuk}.

Let ${\rho_0,\rho_1}$ be  real numbers satisfying  ${0<\rho_0,\rho_1 \le 1}$,  and let $\CC_0,\CC_1$ be the circles on the complex plane centered in~$0$, respectively,~$1$, of radius~$\rho_0$, respectively,~$\rho_1$:
$$
\CC_1:=\{z\in \C: |z|=\rho_0\}, \qquad \CC_1:=\{z\in \C: |z-1|=\rho_1\}. 
$$
We have 
$$
\CC_0\cap\CC_1=
\begin{cases}
\varnothing,&\text{when ${\rho_0+\rho_1 <1}$},\\
\{\rho_0\}, & \text{when ${\rho_0+\rho_1 =1}$}, \\
\{\varsigma,\bar\varsigma\}, & \text{when ${\rho_0+\rho_1 >1}$}, 
\end{cases}
$$
where
\begin{equation}
\label{evars}
\varsigma:= \frac{1+\rho_0^2-\rho_1^2+i\sqrt{4\rho_0^2\rho_1^2-(1-\rho_0^2-\rho_1^2)^2}}{2}. 
\end{equation}
We choose the positive value of the square root, so that ${\im\varsigma >0}$.
See Figure~\ref{ficircles} for an illustration. 

\begin{figure}
\caption{two circles}
\label{ficircles}
\begin{center}
\begin{picture}(90,60)(-30,-30)
\put(-25,0){\vector(1,0){80}}
\put(-2,-6){\tiny{$0$}}
\put(23,-6){\tiny{$1$}}
\put(0,0){\circle*{3}}
\put(25,0){\circle*{3}}
\put(0,0){\circle{25}}
\put(25,0){\circle{15}}
\put(-0,-30){\tiny{${\rho_0+\rho_1 <1}$}}
\end{picture}
\begin{picture}(90,60)(-30,-30)
\put(-25,0){\vector(1,0){80}}
\put(-2,-6){\tiny{$0$}}
\put(23,-6){\tiny{$1$}}
\put(0,0){\circle*{3}}
\put(25,0){\circle*{3}}
\put(0,0){\circle{29}}
\put(25,0){\circle{19}}
\put(14.5,0){\circle*{3}}
\put(12,-11){\tiny{$\rho_0$}}
\put(0,-30){\tiny{${\rho_0+\rho_1 =1}$}}
\end{picture}
\begin{picture}(100,60)(-30,-30)
\put(-25,0){\vector(1,0){80}}
\put(-2,-6){\tiny{$0$}}
\put(23,-6){\tiny{$1$}}
\put(0,0){\circle*{3}}
\put(25,0){\circle*{3}}
\put(0,0){\circle{30}}
\put(25,0){\circle{40}}
\put(10,13){\circle*{3}}
\put(10,-13){\circle*{3}}
\put(8,16){\scriptsize{$\varsigma$}}
\put(8,-20){\scriptsize{$\bar\varsigma$}}
\put(0,-30){\tiny{${\rho_0+\rho_1 >1}$}}
\end{picture}
\end{center}
\end{figure}

\begin{remark}
Equation~\eqref{evars} looks a bit scary, but it is not used in the proof of  Proposition~\ref{pbeuk} below.  
\end{remark}

The following proposition is inspired by  Beukers. 


\begin{proposition}
\label{pbeuk}
Let ${z_0\in \CC_0}$ and ${z_1\in \CC_1}$.
Then
\begin{equation}
\label{ebeuk}
|z_0-z_1|\ge 
\begin{cases}
1-\rho_0-\rho_1 & \text{if $\rho_0+\rho_1<1$},\\ 
&\\
\displaystyle\frac{1}{2\rho_0}|z_0-\rho_0|^2  & \text{if $\rho_0+\rho_1=1$},\\ 
&\\
\displaystyle\frac23 \left(1+ \frac{2\rho_0}{\im\varsigma}\right)^{-1/2}\min \{|z_0-\varsigma|, |z_0-\bar\varsigma|\}  &\text{if $\rho_0+\rho_1>1$}. 
\end{cases}
\end{equation}
\end{proposition}

Informally,  in case  ${\rho_0+\rho_1<1}$ the points~$z_0$ and~$z_1$ cannot be close to each other, and in case  ${\rho_0+\rho_1\ge 1}$,
if they are close to each other, then they must be close to an intersection point of the circles $\CC_0$,~$\CC_1$.  Geometrically it looks clear  and all we need is an explicit estimate.
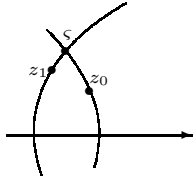
\begin{center}
\begin{picture}(100,80)(-30,-20)
\put(-25,0){\vector(1,0){70}}
\qbezier(-15,0)(-15,30)(20,50)
\qbezier(-15,0)(-14,-10)(-12,-15)
\qbezier(10,0),(10,20),(-10,40)
\qbezier(10,0)(9,-10)(8,-12)
\put(-3,32){\circle*{3}}
\put(-8,25){\circle*{3}}
\put(6,17){\circle*{3}}
\put(-4,37){\scriptsize{$\varsigma$}}
\put(-17,23){\scriptsize{$z_1$}}
\put(6,20){\scriptsize{$z_0$}}
\end{picture}
\end{center}

The proof uses the following lemma.

\begin{lemma}
\label{llere}
Let ${\eta, \rho>0}$ and let~$\CC$  be the circle with center~$0$ and radius~$\rho$. Then for ${z,w\in \CC}$ with 
${|\im z+\im w|\ge\eta}$ we have 
$$
|z-w|\le \left(1+ \frac{2\rho}{\eta}\right)^{1/2}|\re z-\re w|. 
$$
\end{lemma}

\begin{proof}
Since 
${(\re z)^2+(\im z)^2=(\re w)^2+(\im w)^2=\rho^2}$,  
we have 
\begin{equation*}
|\im z - \im w| \le  \frac{1}{\eta}|(\im z)^2-(\im w)^2|= \frac{1}{\eta}\bigl|(\re z)^2-(\re w)^2\bigr|\le \frac{2\rho}{\eta}|\re z-\re w|. 
\end{equation*}
Hence 
$$
|z-w|^2 = |\re z-\re w|^2+|\im z - \im w|^2 \le \left(1+\frac{2\rho}{\eta}\right)|\re z-\re w|^2, 
$$
as wanted. 
\end{proof}

\begin{proof}[Proof of Proposition~\ref{pbeuk}]
If ${\rho_0+\rho_1<1}$ then the result is obvious. Now assume that ${\rho_0+\rho_1\ge1}$. 
Since 
$$
|z_0-z_1|\ge \bigl||z_0-1|-|z_1-1|\bigr|=\bigl||z_0-1|-\rho_1\bigr|
$$
it suffices to bound ${\bigl||z_0-1|-\rho_1\bigr|}$ from below. 

If ${\rho_0+\rho_1=1}$ then 
$$
|z_0-1|^2-\rho_1^2=2(\rho_0-\re z_0) = \frac{1}{\rho_0}|z_0-\rho_0|^2. 
$$
Since ${|z_0-1|+\rho_1\le |z_0|+1+\rho_1=2}$, this implies that 
$$
\bigl||z_0-1|-\rho_1\bigr| \ge \frac{1}{2\rho_0}|z_0-\rho_0|^2, 
$$
which proves~\eqref{ebeuk} in the case ${\rho_0+\rho_1=1}$.

Now assume that ${\rho_0+\rho_1>1}$. Without loss of generality, ${\im z_0 \ge 0}$. 
 We have 
 $$
 |z_0-1|^2= \rho_0^2+1-2\re z_0, \qquad \rho_1^2=|\varsigma-1|^2=  \rho_0^2+1-2\re\varsigma, 
 $$
which implies that 
$$
\bigl||z_0-1|^2-\rho_1^2\bigr|= 2|\re z_0 - \re\varsigma|.
$$
Since ${\im z_0\ge 0}$ and ${\im\varsigma>0}$, Lemma~\ref{llere} applies with ${\eta:=\im\varsigma}$, and we obtain 
$$
|\re z_0 - \re\varsigma| \ge  \left(1+ \frac{2\rho_0}{\im\varsigma}\right)^{-1/2}|z_0-\varsigma|. 
$$
Also, since ${|z_0-1|+\rho_1\le |z_0|+1+\rho_1\le 3}$, we have 
$$
\bigl||z_0-1|-\rho_1\bigr| \ge \frac13 \bigl||z_0-1|^2-\rho_1^2\bigr|.
$$
Putting all this together, we obtain 
$$
\bigl||z_0-1|-\rho_1\bigr| \ge \frac23 \left(1+ \frac{2\rho_0}{\im\varsigma}\right)^{-1/2}|z_0-\varsigma|, 
$$
which proves~\eqref{ebeuk} in  case ${\rho_0+\rho_1>1}$ as well. 

\end{proof}

\section{The principal argument}
\label{sthree}

In this section~$K$ 
is a number field~$K$ of degree~$d$.  The following three statements play the central role in the proof of Theorem~\ref{thmstv}.  

\begin{proposition}
\label{prtwo}
Let ${\alpha_1,\alpha_2, \lambda_1,\lambda_2\in K^\times}$, and let ${v\in M_K}$ be an absolute value with underlying prime~$p$ (including ${p=\infty}$). Assume that 
${\lambda_1/\lambda_2}$ is not a root of unity. Denote ${h=\max\{\height(\alpha_1),\height(\alpha_2)\}}$. 
Then there exists ${C>0}$, effectively depending on $\lambda_1$,~$\lambda_2$,~$d$ and~$p$, such that for any ${n\in \Z}$  with ${n\ge 3}$ the following holds: either ${n\le Ch}$ or 
$$
|\alpha_1\lambda_1^n+\alpha_2\lambda_2^n|_v\ge |\lambda_1|_v^n e^{-C(h+1)\log n} .
$$
\end{proposition}

\begin{proposition}
\label{prunits}
Let ${\alpha_1,\alpha_2,\alpha_3 , \lambda_1,\lambda_2, \lambda_3\in K^\times}$. Assume that one of the quotients ${\lambda_j/\lambda_k}$ is not a root of unity. 
Denote
\begin{equation}
\label{emax}
h:=\max\{\height(\alpha_1),\height(\alpha_2),\height(\alpha_3)\}
\end{equation}
Let ${n\in \Z}$ be such that 
\begin{equation}
\label{eunit}
\alpha_1\lambda_1^n+\alpha_2\lambda_2^n+\alpha_3\lambda_3^n=0. 
\end{equation}
Then 
\begin{equation}
\label{eunits}
|n| \le C(h+1)\log (h+2),
\end{equation}
where~$C$ effectively depends on $\lambda_1$,~$\lambda_2$,~$\lambda_3$ and~$d$. 
\end{proposition}

\begin{proposition}
\label{prbeuk}
Let ${\alpha_1,\alpha_2,\alpha_3 , \lambda_1,\lambda_2, \lambda_3\in K^\times}$, and let~$h$ be defined by~\eqref{emax}. 
Assume that $\lambda_j/\lambda_k$ is not a root of unity\footnote{The hypothesis that $\lambda_j/\lambda_k$ are not roots of unity can be suppressed, but this would complicate the proof, so we prefer to keep it.} for ${1\le j<k\le3}$.
Let ${v\in M_K}$ be an \textbf{infinite} absolute value of~$K$ such that 
$$
|\lambda_1|_v=|\lambda_2|_v=|\lambda_3|_v.
$$ 
Let ${n\in \Z}$ with ${n\ge 3}$ be such that 
\begin{equation}
\label{enonz}
\alpha_1\lambda_1^n+\alpha_2\lambda_2^n+\alpha_3\lambda_3^n\ne 0. 
\end{equation}
Then 
$$
|\alpha_1\lambda_1^n+\alpha_2\lambda_2^n+\alpha_3\lambda_3^n|_v\ge |\lambda_1|_v^n e^{-C(h+1)\log n}, 
$$
where~$C$ effectively depends on $\lambda_1$,~$\lambda_2$,~$\lambda_3$ and~$d$. 
\end{proposition}

All three propositions can be found, mutatis mutandis, in both \cite{MST84} and~\cite{Ve85}. The first two were known well before; in particular, Proposition~\ref{prunits} is a very special case of a classical result about binary unit equations: if~$\Gamma$ is a finitely generated multiplicative group of algebraic numbers, and ${\alpha,\beta\in \bar\Q^\times}$, then equation ${\alpha x+\beta y=1}$ has at most finitely many solutions in ${x,y\in \Gamma}$, and the heights of these solutions can be explicitly bounded in terms of~$\Gamma$ and the heights of~$\alpha$ and~$\beta$. The original contribution of~\cite{MST84,Ve85} is Proposition~\ref{prbeuk}. 

\begin{proof}[Proof of Proposition~\ref{prtwo}] Write 
$$
|\alpha_1\lambda_1^n+\alpha_2\lambda_2^n|_v= |\alpha_1||\lambda_1|_v^n |\alpha\lambda^n-1|_v, 
$$
where ${\alpha=-\alpha_2/\alpha_1}$ and ${\lambda:=\lambda_2/\lambda_1}$. By the hypothesis,~$\lambda$ is not a root of unity, and we have ${\height(\alpha) \le 2h}$ by Proposition~\ref{prheights}. 

If ${\alpha\lambda^n=1}$ then ${n\le C_1h}$ by Corollary~\ref{contoh}. Otherwise, 
$$
|\alpha\lambda^n-1|_v\ge e^{-C_2(h+1)\log n}
$$
by Proposition~\ref{prbak}. We also have ${|\alpha_1|_v \ge e^{-dh}}$ by  Liouville's inequality (Proposition~\ref{prliou}). The proof is complete with ${C:=\max\{C_1,C_2+d\}}$. 
\end{proof}

For Proposition~\ref{prunits} we need a lemma.  By the Kronecker Theorem (see Proposition~\ref{prheights}), if ${\gamma\in K^\times }$ is not a root of unity then there exist ${v,v'\in M_K}$ such that ${|\gamma|_v<1}$ and ${|\gamma|_{v'}>1}$. We need a quantitative version of this. 

\begin{lemma}
\label{lchoose}
There exists a  ${\theta=\theta(d)>0}$, depending only on~$d$, such that the following holds. 
Let ${\gamma \in K^\times}$ be not a root of unity. Then for some ${v \in M_K}$ we have ${|\gamma|_v \le e^{-\theta}}$. 
\end{lemma}

\begin{proof}
Since~$\gamma$ is not a root of unity, there must exist ${v \in M_K}$  with ${|\gamma|_v <1}$. If~$v$ is finite,~$p$  its underlying prime and~$e$ the ramification index,  then  
$$
|\gamma|_v \le p^{-1/e}\le 2^{-1/d}.
$$ 
Now assume that ${|\gamma|_v=1}$ for all finite~$v$. Then
$$
d\height(\gamma)=d\height(\gamma^{-1})= \sum_{v\mid \infty}d_v\log \max\{1,|\gamma|_v^{-1}\}. 
$$
Since ${\sum_{v\mid\infty}d_v=d}$, this implies that there exists ${v\mid \infty}$ such that ${|\gamma|_v \le e^{-\height(\gamma)}}$.  In particular, 
${|\gamma|_v \le e^{-\kappa(d)}}$, see Item~\ref{ikr} of Proposition~\ref{prheights}. This proves the lemma with 
${\theta:=\min \left\{(\log2)/d, \kappa(d)\right\}}$. 
As follows from~\eqref{evol}, 
${\theta:=1/2d(\log(3d))^3}$
would do for all~$d$. 
\end{proof}

\begin{proof}[Proof of Proposition~\ref{prunits}]
We may assume that ${n\ge 3}$, replacing~$n$ by~$-n$, if necessary. 
Dividing by $\alpha_3\lambda_3^n$ and redefining $\alpha_j, \lambda_j$, we rewrite equation~\eqref{eunit} as 
$$
\alpha_1\lambda_1^n-1=\alpha_2\lambda_2^n. 
$$
(Note that with the newly defined $\alpha_j, \lambda_j$ we must replace~$h$ by~$2h$.)
By the hypothesis, we may assume that~$\lambda_2$ is not a root of unity. Hence there exists ${v\in M_K}$ with ${|\lambda_2|_v<e^{-\theta}}$, where~$\theta$  
is from Lemma~\ref{lchoose}. 

Proposition~\ref{prbak} implies that ${|\alpha_1\lambda_1^n-1|_v \ge e^{-C_1(h+1)\log n}}$, where~$C_1$ depends on~$\lambda_j$'s, on~$d$, and, if~$v$ is finite, on the underlying prime~$p$. However, in the latter case ${|\lambda_2|_v \le p^{-1/e}\le p^{-1/d}}$, which implies that ${\log p \le \height(\lambda_2)}$, bounding~$p$ in terms of~$\lambda_2$. Hence~$C_1$ depends only on~$\lambda_j$'s and~$d$. 

Thus, 
${e^{-C_1(h+1)\log n} \le e^{-\theta n}}$. Hence ${n \le C(h+1)\log(h+2)}$, as wanted. 
\end{proof}

Using a more precise version of Baker's inequality, one can get rid of the logarithmic factor in~\eqref{eunits}. But for our purposes we are happy with~\eqref{eunits} as it stands. 

Proposition~\ref{prbeuk} is bit more sophisticated; its proof is based on an elementary  geometric trick, credited in~\cite{MST84} to Beukers~\cite{BT84}, and formalized in Proposition~\ref{pbeuk} above. 

\begin{proof}[Proof of Proposition~\ref{prbeuk}]
We may assume that ${K\subset \C}$ and write the absolute value simply as ${|\cdot|}$. In particular, our hypothesis is 
\begin{equation}
\label{ehypo}
|\lambda_1|=|\lambda_2|=|\lambda_3|.
\end{equation}
Denote  the left-hand side of~\eqref{enonz} by $A(n)$, and write
$$
A(n)=\alpha_1\lambda_1^n(z_1-z_0),
$$
where 
$$
z_0:= -\frac{\alpha_2}{\alpha_1}\left(\frac{\lambda_2}{\lambda_1}\right)^n, \qquad z_1:= 1+ \frac{\alpha_3}{\alpha_1}\left(\frac{\lambda_3}{\lambda_1}\right)^n. 
$$
Since ${|\alpha_1|_v \ge e^{-dh}}$ by  Liouville's inequality, we only have to show that 
\begin{equation}
\label{ewhat}
|z_1-z_0|\ge e^{-C_1(h+1)\log n},
\end{equation}
where $C_1,C_2, \ldots$ denote in this proof positive numbers effectively depending on the $\lambda_j$'s and on~$d$. 

We set ${\rho_0:=|\alpha_2/\alpha_1|}$ and ${\rho_1:=|\alpha_3/\alpha_1|}$. Our hypothesis~\eqref{ehypo} implies that ${z_0\in \CC_0}$ and ${z_1\in \CC_1}$, where ${\CC_0,\CC_1}$ are as in Proposition~\ref{pbeuk}.  In the sequel we will deal with the complex algebraic numbers 
\begin{equation}
\label{enumbers}
\rho_0, \quad \rho_1, \quad 1-\rho_0-\rho_1, \quad \varsigma, \quad \left(1+ \frac{2\rho_0}{\im\varsigma}\right)^{-1/2}, 
\end{equation}
where~$\varsigma$ is defined in~\eqref{evars}.  Their degrees over~$K$ is bounded by~$C_2$, and, using Proposition~\ref{prheights}, we bound their heights by $C_3(h+1)$. If~$z$ is one of the numbers~\eqref{enumbers} and ${z\ne 0}$ then 
$$
e^{-C_4(h+1)}\le |z|\le e^{C_4(h+1)}
$$
by  Liouville's inequality. 

As in Proposition~\ref{pbeuk}, we have three cases. If ${\rho_0+\rho_1<1}$ then 
$$
|z_1-z_0|\ge |1-\rho_0-\rho_1|\ge e^{-C_4(h+1)}. 
$$
If  ${\rho_0+\rho_1=1}$ 
then 
$$
|z_1-z_0|\ge \frac{1}{2\rho_0}|z_0-\rho_0|^2 = \frac{\rho_0}{2}|\rho_0^{-1}z_0-1|^2. 
$$
We have ${|\rho_0|\ge e^{-C_4(h+1)}}$ and, when ${z_0\ne \rho_0}$, we have
$$
|\rho_0^{-1}z_0-1| = \left|-\rho_0^{-1}  \frac{\alpha_2}{\alpha_1}\left(\frac{\lambda_2}{\lambda_1}\right)^n-1\right| \ge e^{-C_5(h+1)\log n}
$$
by Proposition~\ref{prbak} applied with ${\alpha:=-\rho_0^{-1}  \alpha_2/\alpha_1}$ and ${\lambda:=\lambda_2/\lambda_1}$. (Note that, in general,~$\alpha$ does not belong to~$K$, and parameter~$d$ has to be adjusted while applying Proposition~\ref{prbak}.)  This proves~\eqref{ewhat} in the case ${z_0\ne \rho_0}$. If ${z_0= \rho_0}$ then  ${\alpha\lambda^n = 1}$ and ${n\le C_6h}$ by Corollary~\ref{contoh} (recall that ${\lambda=\lambda_2/\lambda_1}$ is not a root of unity by the hypothesis), which bounds the height of $A(n)$ by ${e^{C_7(h+1)}}$. Now  Liouville's inequality gives ${|A(n)|\ge e^{-dC_7(h+1)}}$. 

The case ${\rho_0+\rho_1>1}$ is very similar. Assuming that ${\im z_0\ge 0}$ (which does not restrict generality), we have 
$$
|z_1-z_0|\ge \frac23 \left(1+ \frac{2\rho_0}{\im\varsigma}\right)^{-1/2}|z_0-\varsigma| .
$$
In the case ${z_0\ne \varsigma}$ we have to bound ${|\varsigma^{-1}z_0-1|}$ from below, which is done using Proposition~\ref{prbak} exactly as above. The case ${z_0= \varsigma}$ is treated exactly as the case ${z_0=\rho_0}$ above. 
\end{proof}

\section{Proof of Theorem~\ref{thmstv}}

We write 
$$
U(n)=f_1(n) \lambda_1^n+\cdots + f_s(n)\lambda_s^n; 
$$
expanding the number field~$K$, we may assume that it contains the coefficients of the polynomials $f_1, \ldots, f_s$. 

We may exclude from consideration the values of~$n$ with ${f_j(n)=0}$ for some~$j$; these~$n$ are finite in number and can be easily listed effectively. Thus, we will tacitly assume that
$$
f_j(n) \ne 0 \qquad (j=1, \ldots, s). 
$$
We have ${\height(n)=\log n}$ for a positive integer~$n$. Proposition~\ref{prheights} implies that 
\begin{equation}
\label{ehfj}
\height(f_j(n))\le C_1\log n, \qquad (j=1, \ldots, s), 
\end{equation}
where $C_1,C_2, \ldots$ denote in this proof positive numbers effectively depending on the $\lambda_j$'s, on the coefficients of the polynomials $f_j$,  and on~$d$. 

Let ${v\in M_K}$ be such that~$U$ has at most~$2$ dominant roots with respect to~$v$. This means that, after renumbering,
$$
|\lambda_1|_v \ge\cdots \ge |\lambda_s|_v, \qquad |\lambda_1|_v >|\lambda_3|_v . 
$$
Denote 
$$
A(n) := f_1(n)\lambda_1^n+f_2(n)\lambda_2^n, \qquad B(n):= f_3(n)\lambda_3^n+\cdots+f_s(n)\lambda_s^n. 
$$
Applying Proposition~\ref{prtwo} with ${\alpha_j:=f_j(n)}$, and taking~\eqref{ehfj} into account, we obtain that either ${n\le C_2\log n}$, in which case~$n$ is bounded, or 
\begin{equation}
\label{eage}
|A(n)|_v \ge |\lambda_1|_v^n e^{-C_3(\log n)^2}. 
\end{equation}
Estimating trivially 
\begin{equation}
\label{eble}
|B(n)|_v \le C_4n^D|\lambda_3|_v^n, 
\end{equation}
where~$D$ is the maximum of degrees of the polynomials $f_j$, we see that 
\begin{equation}
\label{eageb}
|A(n)|_v>|B(n)|_v
\end{equation} 
for sufficiently large~$n$. This complete the proof in the case of~$2$ dominant roots.

Now assume that there are exactly~$3$ dominant roots for some infinite~$v$. The argument is very similar, but now we will be using Propositions~\ref{prunits} and~\ref{prbeuk}.  We have now
$$
|\lambda_1|_v=|\lambda_2|_v=|\lambda_3|_v>|\lambda_4|_v \ge\cdots \ge |\lambda_s|_v, 
$$
and we set
$$
A(n) := f_1(n)\lambda_1^n+f_2(n)\lambda_2^n+f_3(n)\lambda_3^n, \qquad B(n):= f_4(n)\lambda_4^n+\cdots+f_s(n)\lambda_s^n. 
$$
If ${A(n)=0}$ then ${n\le C_5\log n\log\log n}$ by Proposition~\ref{prunits}, which bounds~$n$. Otherwise, we have~\eqref{eage} by Proposition~\ref{prbeuk}, and we again have~\eqref{eble}, but with~$\lambda_3$ replaced by~$\lambda_4$. Hence~\eqref{eageb} again holds for sufficiently large~$n$, and we are done.

\section{Bacik's contribution: proof of Theorem~\ref{thba} }

Bacik~\cite{Ba24} proved the following beautiful lemma. 


\begin{lemma}[Bacik]
\label{lba}
Let~$K$ be a number field and let ${\lambda_1, \lambda_2,\lambda_3,\lambda_4\in K^\times}$ be such that  one of the quotients $\lambda_j/\lambda_k$ is not root of unity. Then there exists ${v\in M_K}$ such that, after renumbering of ${\lambda_1, \ldots, \lambda_4}$, we have one of the following:
\begin{align*}
&|\lambda_1|_v \ge|\lambda_2|_v\ge|\lambda_3|_v\ge |\lambda_4|_v, \qquad |\lambda_1|_v >|\lambda_3|_v ; \\
&|\lambda_1|_v =|\lambda_2|_v=|\lambda_3|_v> |\lambda_4|_v, \qquad v \mid \infty. 
\end{align*}
\end{lemma}

\begin{proof}
Assume the contrary: for every infinite ${v \in M_K}$ we have 
\begin{equation}
\label{ealleq}
|\lambda_1|_v =|\lambda_2|_v=|\lambda_3|_v= |\lambda_4|_v, 
\end{equation}
and  for every finite~$v$ we have 
${|\lambda_1|_v =|\lambda_2|_v=|\lambda_3|_v\ge |\lambda_4|_v}$
after renumbering. 

We may suppose that~$K$ is contained in~$\C$. 
By the assumption, we have ${|\lambda_1| =\cdots= |\lambda_4|}$, and we rewrite this as 
${\lambda_1\bar\lambda_1= \lambda_2\bar\lambda_2=\lambda_3\bar\lambda_3=\lambda_4\bar\lambda_4}$.

Let~$v$ be a finite absolute value and~$\barv$ the conjugate absolute value, defined by ${|x|_\barv:=|\bar x|_v}$.  Then 
\begin{equation}
\label{ebacik}
|\lambda_1|_v |\lambda_1|_\barv=|\lambda_2|_v |\lambda_2|_\barv=|\lambda_3|_v |\lambda_3|_\barv=|\lambda_4|_v |\lambda_4|_\barv. 
\end{equation}
By the assumption, the~$4$ numbers ${|\lambda_1|_v,\ldots,|\lambda_4|_v}$  have the following property:~$3$ of them are equal, and the fourth one is less than or equal to the others.  The same applies to the~$4$
numbers ${|\lambda_1|_\barv,\ldots,|\lambda_4|_\barv}$. Together with~\eqref{ebacik} this implies that all the~$4$ numbers ${|\lambda_1|_v,\ldots,|\lambda_4|_v}$ are equal. 

We have proved the following:~\eqref{ealleq} holds for any ${v\in M_K}$, both infinite and finite.  By the Kronecker Theorem, all the quotients ${\lambda_j/\lambda_k}$ must be roots of unity, a contradiction. 
\end{proof}

Theorem~\ref{thba}  is an immediate consequence of Theorem~\ref{thmstv} and  Lemma~\ref{lba}. 

\paragraph{Acknowledgment} I thank Piotr Bacik for helpful comments.

{\footnotesize

\bibliographystyle{amsplain}
\bibliography{domin}

\noindent
Yuri BILU\\
IMB, Université de Bordeaux \& CNRS\\
\url{yuri@math.u-bordeaux.fr}

}

\end{document}